\documentclass{amsart}            

\begin{document}
\title[On some Lie groups as 5-dimensional \ldots]
{On some Lie groups as 5-dimensional almost contact B-metric
manifolds with three natural connections}

\author{Miroslava Ivanova}

\address[Miroslava Ivanova]{Department of Informatics and Mathematics, Trakia University,
Stara Zagora, 6000, Bulgaria,
E-mail: mivanova@uni-sz.bg}

\author{Hristo Manev}

\address[Hristo Manev]{Department of Pharmaceutical Sciences, Medical University of
Plovdiv,
Plovdiv, 4002, Bulgaria;
Departament of Algebra and Geometry, University of Plovdiv,
Plovdiv, 4027, Bulgaria, %
E-mail: hmanev@uni-plovdiv.bg}

\newcommand{\ie}{i.e. }
\newcommand{\Id}{\mathrm{Id}}
\newcommand{\X}{\mathfrak{X}}
\newcommand{\W}{\mathcal{W}}
\newcommand{\F}{\mathcal{F}}
\newcommand{\T}{\mathcal{T}}
\newcommand{\LL}{\mathcal{L}}
\newcommand{\TT}{\mathfrak{T}}
\newcommand{\G}{\mathcal{G}}
\newcommand{\I}{\mathcal{I}}
\newcommand{\M}{(M,\allowbreak{}\ff,\allowbreak{}\xi,\allowbreak{}\eta,\allowbreak{}g)}
\newcommand{\Lf}{(G,\ff,\xi,\eta,g)}
\newcommand{\R}{\mathbb{R}}
\newcommand{\N}{\widehat{N}}
\newcommand{\s}{\mathfrak{S}}
\newcommand{\n}{\nabla}
\newcommand{\nn}{\tilde{\nabla}}
\newcommand{\ff}{\varphi}
\newcommand{\D}{{\rm d}}
\newcommand{\id}{{\rm id}}
\newcommand{\al}{\alpha}
\newcommand{\bt}{\beta}
\newcommand{\gm}{\gamma}
\newcommand{\dt}{\delta}
\newcommand{\lm}{\lambda}
\newcommand{\ta}{\theta}
\newcommand{\om}{\omega}
\newcommand{\ea}{\varepsilon_\alpha}
\newcommand{\eb}{\varepsilon_\beta}
\newcommand{\eg}{\varepsilon_\gamma}
\newcommand{\sx}{\mathop{\mathfrak{S}}\limits_{x,y,z}}
\newcommand{\norm}[1]{\left\Vert#1\right\Vert ^2}
\newcommand{\nf}{\norm{\n \ff}}
\newcommand{\nN}{\norm{N}}
\newcommand{\Span}{\mathrm{span}}
\newcommand{\grad}{\mathrm{grad}}
\newcommand{\thmref}[1]{The\-o\-rem~\ref{#1}}
\newcommand{\propref}[1]{Pro\-po\-si\-ti\-on~\ref{#1}}
\newcommand{\secref}[1]{\S\ref{#1}}
\newcommand{\lemref}[1]{Lem\-ma~\ref{#1}}
\newcommand{\dfnref}[1]{De\-fi\-ni\-ti\-on~\ref{#1}}
\newcommand{\corref}[1]{Corollary~\ref{#1}}



\numberwithin{equation}{section}
\newtheorem{thm}{Theorem}[section]
\newtheorem{lem}[thm]{Lemma}
\newtheorem{prop}[thm]{Proposition}
\newtheorem{cor}[thm]{Corollary}
\newtheorem{defn}{Definition}[section]

\hyphenation{Her-mi-ti-an ma-ni-fold ah-ler-ian}

\begin{abstract}
Almost contact manifolds with B-metric are considered. 
There are studied three natural connections (i.e. linear connections preserving the
structure tensors) determined by conditions for their torsions. These connections
are investigated on a family of Lie groups considered as 5-dimensional almost
contact B-metric manifolds.
\end{abstract}


\maketitle

\section*{Introduction}

In differential geometry of manifolds with additional tensor
structure, there are important the so-called \emph{natural
connections}, \ie linear connections with respect to which the
structure tensors of the manifolds are parallel.

The natural connections are studied in the geometry of manifolds
with some additional structures: almost Hermitian structure \cite{Ko-No},
almost contact metric structure \cite{Ale-Gan2},
almost complex structure with Norden metric \cite{Gan-Mi}.

The geometry of almost contact B-metric manifolds is the geometry
of the structures $\ff,\xi,\eta,g$ and $\tilde{g}$. For this
geometry a natural connection $D$ is parallel with respect to
$\ff$ and $g$ (consequently, $\xi,\eta$ and $\tilde{g}$). By
$\F_0$ is denoted the class of the considered manifolds with
parallel structure $\ff$ with respect to the Levi-Civita
connection $\n$. The natural connections play the same role
outside the class $\F_0$ such a role plays $\n$ in $\F_0$.

Three natural connections ($\ff$KT-connection, $\ff$B-connection
and $\ff$-canonical connection) on almost contact B-metric
manifolds are an object of special interest in this paper. The
goal of the present work is the investigation of those three
natural connections on a family of Lie groups as 5-dimensional
almost contact B-metric manifolds belonging to a basic class.

The paper is organized as follows. In Sec.~\ref{mi:sec2}, we give
necessary facts about considered manifolds. In Sec.~\ref{mi:sec3},
we consider an example on a family of Lie groups as 5-dimensional
$\F_7$-manifolds with a parallel torsion of the
$\ff$KT-connection. In Sec.~\ref{mi:sec4}, we describe the
$\ff$B-connection and the $\ff$-canonical connection on these
manifolds.

\section{Almost contact manifolds with B-metric}\label{mi:sec2}

Let $(M,\ff,\xi,\eta,g)$ be a $(2n+1)$-dimensional \emph{almost
contact B-met\-ric manifold}. In details, $(\ff,\xi,\eta)$ is an
almost contact structure determined by a tensor field $\ff$ of
type (1,1), a vector field $\xi$ and its dual 1-form $\eta$ as
follows:
\begin{equation*}
\ff\xi=0, \quad \ff^2 = -\Id + \eta \otimes \xi, \quad
\eta\circ\ff=0, \quad \eta(\xi)=1,
\end{equation*}
where $\Id$ is the identity.
Moreover, $g$ is a pseudo-Riemannian metric such that for
any differentiable
vector fields $x$, $y$ on $M$, it is valid: \cite{GaMiGr}
\begin{equation*}
g(\ff x,\ff y)=-g(x,y)+\eta(x)\eta(y).
\end{equation*}

Further, $x$, $y$, $z$, $w$ will stand for arbitrary elements of
$\X(M)$ or vectors of the tangent space at arbitrary point in $M$.

Let us remark that the restriction of a B-metric on the contact
distribution $H=\ker(\eta)$ coincide with the corresponding Norden metric with respect
to the almost complex structure (the restriction of $\ff$ on $H$)
acting as an anti-isometry on the metric (the restriction of $g$) of $H$.
Thus, it is obtained  a correlation
between a $(2n+1)$-dimensional almost contact B-metric manifold
and a $2n$-dimensional almost complex manifold with Norden metric
(or an $n$-dimensional complex Riemannian manifold).

The associated metric $\tilde{g}$ of $g$ on $M$ given by \(
\tilde{g}(x,y)=g(x,\ff y)+\eta(x)\eta(y)\) is a B-metric, too. The
manifold $(M,\ff,\xi,\eta,\tilde{g})$ is also an almost contact
B-metric manifold. Both metrics $g$ and $\tilde{g}$ are indefinite
of signature $(n+1,n)$.

The structure group of $\M$ is $\G\times\I$, where $\I$ is the
identity on $\Span(\xi)$ and $\G=\mathcal{GL}(n;\mathbb{C})\cap
\mathcal{O}(n,n)$.

Let $\n$ be the Levi-Civita connection of $g$. The tensor $F$ of
type (0,3) on $M$ is defined by $F(x,y,z)=g\bigl( \left( \n_x \ff
\right)y,z\bigr)$. It has the following properties:
\begin{equation*}\label{mi:F-prop}
F(x,y,z)=F(x,z,y)
=F(x,\ff y,\ff z)+\eta(y)F(x,\xi,z)
+\eta(z)F(x,y,\xi).
\end{equation*}

A classification of the almost contact B-metric manifolds is
introduced in \cite{GaMiGr}, where eleven basic classes
$\F_i$ $(i=1,\allowbreak{}2,\allowbreak{}\dots,\allowbreak{}{11})$
of these manifolds are characterized with respect to $F$. These
basic classes intersect in the special class $\F_0$ determined by
the condition $F(x,y,z)=0$. Hence $\F_0$ is the class of almost
contact B-metric manifolds with $\n$-parallel structures, \ie
$\n\ff=\n\xi=\n\eta=\n g=\n \tilde{g}=0$.

Let $g_{ij}$,
$i,j\in\{1,2,\dots,2n+1\}$, are the components of the matrix of $g$ with respect
to a basis $\{e_i\}_{i=1}^{2n}=\{e_1,e_2,\dots,e_{2n+1}\}$ of the tangent
space $T_pM$ of $M$ at an arbitrary point $p\in M$, and $g^{ij}$
-- the components of the inverse  matrix of  $(g_{ij})$.

It is defined the \emph{square norm of $\nabla \ff$}
by: \cite{Man31}
\begin{equation*}\label{mi:snf}
    \norm{\nabla \ff}=g^{ij}g^{ks}
    g\bigl(\left(\nabla_{e_i} \ff\right)e_k,\left(\nabla_{e_j}
    \ff\right)e_s\bigr).
\end{equation*}
It is clear, the equality $\norm{\nabla \ff}=0$ is valid if $\M$
is a $\F_0$-manifold, but the inverse implication is not always
true. An almost contact B-metric manifold having a zero square
norm of $\n\ff$ is called an
\emph{isotropic-$\F_0$-manifold} \cite{Man31}.

The Nijenhuis tensor $N$ of the almost contact structure is
defined by $N = [\ff, \ff]+ \D{\eta}\otimes\xi$, where $[\ff,
\ff](x, y)=\left[\ff x,\ff
y\right]+\ff^2\left[x,y\right]-\ff\left[\ff
x,y\right]-\ff\left[x,\ff y\right]$ for
$\left[x,y\right]=\n_xy-\n_yx$ and $\D\eta$ is the exterior
derivative of $\eta$. In \cite{ManIv36}, it is defined an
associated Nijenhuis tensor $\widehat{N}$  by
$\widehat{N}=\{\ff,\ff\}+\left(\LL_{\xi}g\right)\otimes\xi$, where
$\{\ff,\ff\}(x, y)=\ff^2\{x,y\}+\{\ff x,\ff y\}-\ff\{\ff
x,y\}-\ff\{x,\ff  y\}$ for $\left\{x,y\right\}=\n_xy+\n_yx$ and
$\LL_{\xi}g$ is the Lie derivative with respect to $\xi$ of the
metric $g$.

Hence, $N$ and $\widehat{N}$ in terms of the covariant derivatives
has the following form:
\begin{equation*}\label{mi:N, hat N}
\begin{split}
N(x,y)&=\left(\n_{\ff x}\ff\right)y-\ff\left(\n_{x}\ff\right)y+\left(\n_{x}\eta\right)y\cdot\xi\\
& -\left(\n_{\ff
y}\ff\right)x+\ff\left(\n_{y}\ff\right)x-\left(\n_{y}\eta\right)x\cdot\xi,\\
\widehat{N}(x,y)&=\left(\n_{\ff x}\ff\right)y-\ff\left(\n_{x}\ff\right)y
+\left(\n_{x}\eta\right)y\cdot\xi\\
& +\left(\n_{\ff
y}\ff\right)x-\ff\left(\n_{y}\ff\right)x+\left(\n_{y}\eta\right)x\cdot\xi.
\end{split}
\end{equation*}

The corresponding tensors of type (0,3) on $\M$ are determined by
$N(x,y,z)=g\left(N(x,y),z\right)$ and $\widehat{N}(x,y,z)=g\left(
\widehat{N}(x,y),z\right)$.

In \cite{Man-Gri2}, a tensor $L$ of type (0,4) on $M$ with
properties
\begin{gather}
    L(x,y,z,w)=-L(y,x,z,w)=-L(x,y,w,z),\label{mi:L-curv-like-1}
    \\
    L(x,y,z,w)+L(y,z,x,w)+L(z,x,y,w)=0 
    \label{mi:L-curv-like-2}
\end{gather}
is called a \emph{curvature-like tensor}.

The curvature tensor $R$ of $\n$ is determined by
$R(x,y)z=\n_x\n_yz-\n_y\n_xz-\n_{[x,y]}z$. The corresponding
tensor of $R$ of type (0,4) is denoted by the same letter and is
defined by the condition $R(x,y,z,w)=g(R(x,y)z,w)$.

The Ricci tensor $\rho$ and the scalar curvature $\tau$ for $R$
as well as their associated quantities are defined by the
traces $\rho(x,y)=g^{ij}R(e_i,x,y,e_j)$,
$\tau=g^{ij}\rho(e_i,e_j)$, $\rho^{*}(x,y)=g^{ij}R(e_i,x,y,\ff
e_j)$ and $\tau^{*}=g^{ij}\rho^{*}(e_i,e_j)$, respectively. In a
similar way there are determined the Ricci tensor $\rho(L)$, the
scalar curvature $\tau(L)$ and their associated
quantities for any curvature-like tensor $L$.

Let $\al$ be a non-degenerate 2-plane (section) in the tangent space $T_pM$, $p\in M$. 
The special 2-planes with respect to the almost
contact B-metric structure $(\ff,\xi,\eta,g)$ are:
\emph{totally real section} if $\al$ is orthogonal to its
$\ff$-image $\ff\al$ and $\xi$,
\emph{$\ff$-holomorphic  section} if $\al$ coincides with
$\ff\al$ and
\emph{$\xi$-section} if $\xi$ lies on $\al$ \cite{NakGri}.

The sectional curvature $k(\al; p)(L)$ of
$\al$ with an arbitrary basis $\{x,y\}$ at $p$
regarding a curvature-like tensor $L$ is defined by
\begin{equation}\label{mi:sec curv}
k(\al; p)(L)=\frac{L(x,y,y,x)}{g(x,x)g(y,y)-g(x,y)^2}.%
\end{equation}

It is known, \cite{Man31}, a linear connection $D$ is
called a \emph{natural connection} on the manifold
$(M,\ff,\allowbreak\xi,\eta,g)$ if the almost contact structure
$(\ff,\xi,\eta)$ and the B-metric $g$ (consequently also
$\tilde{g}$) are parallel with respect to $D$, \ie
$D\ff=D\xi=D\eta=Dg=D\tilde{g}=0$.
%
In \cite{ManIv36}, it is proved that a linear connection
$D$ is natural on $(M,\ff,\allowbreak\xi,\eta,g)$ if and only if
$D\ff=Dg$.
%
A natural connection exists on any almost contact B-metric manifold
and coincides with the Levi-Civita connection only on a
$\F_0$-manifold.

The torsion tensor $T$ of $D$ is determined by
$T(x,y)=D_xy-D_yx-[x,y]$ and the corresponding tensor of type
(0,3) is defined by the condition $T(x,y,z)=g(T(x,y),z)$.

Let $Q$ be the potential tensor of $D$ with respect to $\n$
determined by
\begin{equation}\label{mi:Q}
D_xy=\n_xy+Q(x,y).
\end{equation}
The corresponding tensor of type (0,3) is defined as follows
$Q(x,y,z)=g(Q(x,y),z)$.

According to \cite{Man31}, a linear connection $D$
is a natural connection on an almost contact B-metric manifold if and only if %
\begin{gather*}
 Q(x,y,\ff z)-Q(x,\ff y,z)=F(x,y,z),\label{mi:1a}
 \\%
 Q(x,y,z)=-Q(x,z,y).\label{mi:1b}
\end{gather*}

A natural connection $\ddot{D}$ on the manifold $\M$, which
torsion tensor $\ddot{T}$ is totally skew-symmetric (i.e. a
3-form), exists in the basic classes $\F_3$ and $\F_7$, it is
unique and is called the
\emph{$\ff$KT-connection} \cite{Man31}. %
Let us remark that this connection in the Hermitian geometry is
known as the Bismut connection or the KT-connection.

In \cite{Man-Gri2}, it is introduced a natural connection
$\dot{D}$ on $\M$ in all basic classes by
\begin{equation}\label{mi:defn-fiB}
\begin{array}{l}
\dot{D}_xy=\n_xy+\frac{1}{2}\bigl\{\left(\n_x\ff\right)\ff
y+\left(\n_x\eta\right)y\cdot\xi\bigr\}-\eta(y)\n_x\xi.
\end{array}
\end{equation}
This connection is called a \emph{$\ff$B-connection} in
\cite{ManIv37} and it is studied for the main classes
$\F_1,\F_4,\F_5,\F_{11}$  in Refs.~\cite{Man-Gri2}, \cite{Man3}
and \cite{Man4}. The $\ff$B-connection is the odd-dimensional
counterpart of the B-connection on the corresponding almost
complex manifold with Norden metric, studied for the class $\W_1$
in \cite{GaGrMi}.

In \cite{ManIv38}, a natural connection $\dddot{D}$ is
called a \emph{$\ff$-canonical connection} on
$(M,\ff,\xi,\allowbreak\eta,g)$ if its torsion tensor $\dddot{T}$
satisfies the following identity:
\begin{equation}\label{mi:T-can}
\begin{split}
    &\dddot{T}(x,y,z)-\dddot{T}(x,z,y)-\dddot{T}(x,\ff y,\ff z)
    +\dddot{T}(x,\ff z,\ff y)=\\
    &=\eta(x)\left\{\dddot{T}(\xi,y,z)-\dddot{T}(\xi,z,y)
    -\dddot{T}(\xi,\ff y,\ff z)+\dddot{T}(\xi,\ff z,\ff y)\right\}\\
    &+\eta(y)\left\{\dddot{T}(x,\xi,z)-\dddot{T}(x,z,\xi)
    -\eta(x)\dddot{T}(z,\xi,\xi)\right\}\\
    &-\eta(z)\left\{\dddot{T}
(x,\xi,y)-\dddot{T}(x,y,\xi)-\eta(x)\dddot{T}(y,\xi,\xi)\right\}.
\end{split}
\end{equation}

In this work we pay attention on the case when the manifold $\M$
belongs to the class $\F_7$ from the mentioned classification.
This basic class is characterized by the conditions: \cite{Man8}
\begin{equation*}
\begin{split}
\F_{7}:\quad &F(x,y,z)=F(x,y,\xi)\eta(z)+F(x,z,\xi)\eta(y),\\
                &F(x,y,\xi)=-F(y,x,\xi)=-F(\ff x,\ff y,\xi).
\end{split}
\end{equation*}

The class $\F_3\oplus\F_7$ is characterized by the condition
$\widehat{N}=0$. This is the only class of $\M$ where the
$\ff$KT-connection exists \cite{Man31}. The basic classes $\F_3$
and $\F_7$ are the horizontal component and the vertical one of
$\F_3\oplus\F_7$, respectively. We are interesting in $\F_7$
because the contact distribution $\ker(\eta)$ of any
$\F_3$-manifold is an almost complex manifold with Norden metric
belonging to the class $\W_3$ of quasi-K\"ahler manifolds with
Norden metric which are well-studied in relation with their
curvature properties and natural connections (see e.g.
\cite{MekMan05}, \cite{Mek08a}, \cite{Mek08b},
\cite{Mek09}, \cite{Mek10}).
The mentioned topics on $\F_7$-manifolds are studied in
\cite{Man31}, \cite{Man33} and \cite{ManIv37}. In
\cite{ManIv36}, it is proved that if $\M$ is an arbitrary
manifold in $\F_i, \; i\in\{3,7\}$, then $\ff$B-connection
$\dot{D}$ is the average connection of the $\ff$KT-connection
$\ddot{D}$ and the $\ff$-canonical connection $\dddot{D}$, i.e.
$2\dot{D}=\ddot{D}+\dddot{D}$.

\section{A family of Lie groups as 5-dimensional $\F_7$-manifolds}\label{mi:sec3}

Let us consider the example given in \cite{Man31}. Let $G$
be a 5-dimensional real connected Lie group and $\mathfrak{g}$ be
its Lie algebra. Let $\left\{e_i\right\}_{i=1}^{5}$ be a global
basis of left-invariant vector fields of $G$. An almost contact
B-metric structure is introduced by
\begin{gather}
\begin{array}{c}\label{mi:f}
\ff e_1 = e_3,\quad \ff e_2 = e_4,\quad \ff e_3 =-e_1,\quad \ff
e_4 = -e_2,\quad \ff e_5 =0;\\
\xi=e_5;\qquad \eta(e_1)=\eta(e_2)=\eta(e_3)=\eta(e_4)=0,\quad \eta(e_5)=1;
\end{array}
\\
\begin{array}{c}\label{mi:g}
g(e_1,e_1)=g(e_2,e_2)=-g(e_3,e_3)=-g(e_4,e_4)=g(e_5,e_5)=1,
\\
g(e_i,e_j)=0,\; i\neq j,\quad  i,j\in\{1,2,3,4,5\};
\end{array}
\\
\begin{array}{l}\label{mi:komutator}
\left[e_1,e_2\right]=-\left[e_3,e_4\right]=-\lm_1e_1-\lm_2e_2+\lm_3e_3+\lm_4e_4+2\mu_1e_5,\nonumber\\
\left[e_1,e_4\right]=-\left[e_2,e_3\right]=-\lm_3e_1-\lm_4e_2-\lm_1e_3-\lm_2e_4+2\mu_2e_5.\nonumber
\end{array}
\end{gather}
It is proved that $(G,\ff,\xi,\eta,g)$ belongs to the class
$\F_7$. It is not an $\F_0$-manifold if and only if
$(\mu_1,\mu_2)\neq (0,0)$ holds \cite{Man31}.

Actually, $G$ is a family of manifolds determined by six real
parameters $\lm_1,\dots,\lm_4$, $\mu_1$, $\mu_2$.

In \cite{Man31}, there are obtained the components of the
Levi-Civita connection $\n$ and the $\ff$KT-connection $\ddot{D}$
on the manifold $(G,\ff,\xi,\eta,g)$:
\begin{equation}\label{mi:nabli}
\begin{array}{l}
\n_{e_1}e_1=-\n_{e_3}e_3=\ddot{D}_{e_1}e_1=-\ddot{D}_{e_3}e_3=\lm_1e_2-\lm_3e_4, \\
\n_{e_1}e_2=-\n_{e_3}e_4=-\lm_1e_1+\lm_3e_3+\mu_1e_5, \\
\ddot{D}_{e_1}e_2=-\ddot{D}_{e_3}e_4=-\lm_1e_1+\lm_3e_3, \\
\n_{e_1}e_3=\n_{e_3}e_1=\ddot{D}_{e_1}e_3=\ddot{D}_{e_3}e_1=\lm_3e_2+\lm_1e_4, \\
\n_{e_1}e_4=\n_{e_3}e_2=-\lm_3e_1-\lm_1e_3+\mu_2e_5, \\
\ddot{D}_{e_1}e_4=\ddot{D}_{e_3}e_2=-\lm_3e_1-\lm_1e_3, \\
\n_{e_2}e_1=-\n_{e_4}e_3=\lm_2e_2-\lm_4e_4-\mu_1e_5, \\
\ddot{D}_{e_2}e_1=-\ddot{D}_{e_4}e_3=\lm_2e_2-\lm_4e_4,\\
\n_{e_2}e_2=-\n_{e_4}e_4=\ddot{D}_{e_2}e_2=-\ddot{D}_{e_4}e_4=-\lm_2e_1+\lm_4e_3, \\
\n_{e_2}e_3=\n_{e_4}e_1=\lm_4e_2+\lm_2e_4-\mu_2e_5,\\
\ddot{D}_{e_2}e_3=\ddot{D}_{e_4}e_1=\lm_4e_2+\lm_2e_4,\\
\n_{e_2}e_4=\n_{e_4}e_2=\ddot{D}_{e_2}e_4=\ddot{D}_{e_4}e_2=-\lm_4e_1-\lm_2e_3,\\
\n_{e_1}e_5 =\n_{e_5}e_1 =\frac12\ddot{D}_{e_5}e_1=-\mu_1e_2+\mu_2e_4,\\
\n_{e_2}e_5 =\n_{e_5}e_2 =\frac12\ddot{D}_{e_5}e_2=\mu_1e_1-\mu_2e_3,\\
\n_{e_3}e_5=\n_{e_5 }e_3=\frac12\ddot{D}_{e_5}e_3=-\mu_2e_2-\mu_1e_4,\\
\n_{e_4}e_5=\n_{e_5 }e_4=\frac12\ddot{D}_{e_5}e_4=\mu_2e_1+\mu_1e_3, \\
\n_{e_5}e_5 =\ddot{D}_{e_1}e_5=\ddot{D}_{e_2}e_5=
\ddot{D}_{e_3}e_5=\ddot{D}_{e_4}e_5=\ddot{D}_{e_5}e_5=0.%
\end{array}
\end{equation}

The $\ff$KT-connection $\ddot{D}$ in $\F_7$ is defined
by: \cite{Man31}
\[
\ddot{T}(x,y)=2\{\eta(x)\n_y\xi-\eta(y)\n_x\xi+(\n_x\eta)y\cdot\xi\}.
\]
It is obtained that the basic non-zero components of the torsion
of $\ddot{D}$ are:
\begin{equation*}\label{mi:T-eijk}
\begin{array}{l}
\ddot{T}_{125}=-\ddot{T}_{215}=-\ddot{T}_{345}=\ddot{T}_{435}=-2\mu_1,\\
\ddot{T}_{145}=-\ddot{T}_{235}=\ddot{T}_{325}=-\ddot{T}_{415}=-2\mu_2.
\end{array}
\end{equation*}

For the square norm
$\norm{\ddot{T}}=g^{ip}g^{jq}g^{ks}\ddot{T}_{ijk}\ddot{T}_{pqs}$
we obtain
\begin{equation}\label{mi:norm dot{T}}
\norm{\ddot{T}}=16\left(\mu^2_1-\mu^2_2\right).
\end{equation}

For  $\n$ (respectively, $\ddot{D}$) there are computed the basic
components $R_{ijkl}=R(e_i,e_j,e_k,e_l)$ of the curvature tensor
$R$ (respectively, $\ddot{R}$), $\rho_{jk}=\rho(e_j,e_k)$ of the
Ricci tensor $\rho$ (respectively, $\ddot{\rho}$) and the values
of the scalar curvature $\tau$ (respectively, $\ddot{\tau}$) as
follows (the remaining ones are obtained, according to
\eqref{mi:L-curv-like-1} and \eqref{mi:L-curv-like-2}):
\begin{gather}
\begin{split}\label{mi:R-F7}
R_{1212}&=R_{3434}=\left(\lm^2_1+\lm^2_2-\lm^2_3-\lm^2_4\right)+3\mu^2_1,\\
R_{1234}&=-\left(\lm^2_1+\lm^2_2-\lm^2_3-\lm^2_4\right)-2\mu^2_1+\mu^2_2,\\
R_{1414}&=R_{2323}=-\left(\lm^2_1+\lm^2_2-\lm^2_3-\lm^2_4\right)+3\mu^2_2,\\
R_{1423}&=\left(\lm^2_1+\lm^2_2-\lm^2_3-\lm^2_4\right)+\mu^2_1-2\mu^2_2,\\
R_{1214}&=-R_{1223}=R_{2334}=-R_{1434}=2\left(\lm_1\lm_3+\lm_2\lm_4\right)+3\mu_1\mu_2,\\
R_{1324}&=-\left(\mu^2_1+\mu^2_2\right), \qquad
R_{1535}=R_{2545}=-2\mu_1\mu_2,\\
R_{1515}&=R_{2525}=-R_{3535}=-R_{4545}=-\mu^2_1+\mu^2_2;
\end{split}
\\
\begin{array}{l}\label{mi:ro}
\rho_{11}=\rho_{22}=-\rho_{33}=-\rho_{44}=-2\left(\lm^2_1+\lm^2_2-\lm^2_3-\lm^2_4\right)
-2\left(\mu^2_1-\mu^2_2\right),\\
\rho_{13}=\rho_{24}=-4\left(\lm_1\lm_3+\lm_2\lm_4\right)-4\mu_1\mu_2, \qquad
\rho_{55}=4\left(\mu^2_1-\mu^2_2\right);\\
\end{array}
\\
\begin{array}{l}\label{mi:tao}
\tau=-8\left(\lm^2_1+\lm^2_2-\lm^2_3-\lm^2_4\right)-4\left(\mu^2_1-\mu^2_2\right);\\
\end{array}
\\
\begin{split}\label{mi:K-fB-F7}
\ddot{R}_{1212}&=-\ddot{R}_{1234}=\ddot{R}_{3434}=\left(\lm_1^2+\lm_2^2-\lm_3^2-\lm_4^2\right)+4\mu_1^2,\\
\ddot{R}_{1414}&=-\ddot{R}_{1423}=\ddot{R}_{2323}=-\left(\lm_1^2+\lm_2^2-\lm_3^2-\lm_4^2\right)+4\mu_2^2,\\
\ddot{R}_{1214}&=-\ddot{R}_{1223}=\ddot{R}_{2334}
=-\ddot{R}_{1434}=2\left(\lm_1\lm_3+\lm_2\lm_4\right)+4\mu_1\mu_2;
\end{split}
\\
\begin{array}{l}\label{mi:rofKT-F7}
\ddot{\rho}_{11}=\ddot{\rho}_{22}=-\ddot{\rho}_{33}=-\ddot{\rho}_{44}
=-2\left(\lm^2_1+\lm^2_2-\lm^2_3-\lm^2_4\right)-4\left(\mu^2_1-\mu^2_2\right),\\
\ddot{\rho}_{13}=\ddot{\rho}_{24}=-4\left(\lm_1\lm_3+\lm_2\lm_4\right)-8\mu_1\mu_2;
\end{array}
\\
\begin{array}{l}\label{mi:taofKT-F7}
\ddot{\tau}=-8\left(\lm^2_1+\lm^2_2-\lm^2_3-\lm^2_4\right)-16\left(\mu^2_1-\mu^2_2\right).
\end{array}
\end{gather}

In \cite{Man31}, it is established
\begin{prop}\label{mm:prop:Man31}
The following conditions are equivalent:
\begin{enumerate}
\item The manifold $(G,\ff,\xi,\eta,g)$ is an
isotropic-$\F_0$-manifold; \item The scalar curvatures for $\n$
and $\ddot{D}$ are equal; \item The vectors $\n_{e_i}\xi \;
(i=1,2,3,4)$ are isotropic; \item The equality $\mu_1=\pm\mu_2$ is
valid.
\end{enumerate}
\end{prop}
Moreover, the manifold $(G,\ff,\xi,\eta,g)$ is Einsteinian if and
only if the following conditions are valid: \cite{Man31}
\[
\mu_1\mu_2=-\left(\lm_1\lm_3+\lm_2\lm_4\right), \quad
\mu^2_1-\mu^2_2=-\frac{1}{3}\left(\lm^2_1+\lm^2_2-\lm^2_3-\lm^2_4\right).
\]

\section{A family of Lie groups as 5-dimensional $\F_7$-manifolds --- Further investigations}\label{mi:sec4}

In this section we continue the studying of the manifold $(G,\ff,\xi,\eta,g)$.

In \cite{ManDiss}, there is determined the form of the
Nijenhuis tensor $N$ on  $\M\in\F_7$ as follows:
\begin{equation}\label{mi:N-F7}
N(x,y)=4\left(\n_{x}\eta\right)y
\cdot \xi. 
\end{equation}

According to \eqref{mi:N-F7}, \eqref{mi:f} and \eqref{mi:nabli}, we
obtain
the following non-zero components $N_{ij}=N(e_i,e_j)$ of $N$:
\begin{equation*}\label{mi:Nij}
\begin{array}{l}
N_{12}=-N_{21}=-N_{34}=N_{43}=-4\mu_1\xi,\\
N_{14}=-N_{23}=N_{32}=-N_{41}=-4\mu_2\xi.
\end{array}
\end{equation*}

For the square norm $\nN=g^{ik}g^{js}g(N_{ij},N_{ks})$ we obtain
\begin{equation}\label{mi:norm N}
\nN=64\left(\mu^2_1-\mu^2_2\right).
\end{equation}

Using \eqref{mi:f}, \eqref{mi:g}, \eqref{mi:R-F7} and
\eqref{mi:K-fB-F7}, we compute the components of the associated
Ricci tensors for $\n$ and
$\ddot{D}$, respectively. The non-zero components of $\rho^{*}$ and
associated scalar curvatures are the following:
\begin{gather}
\begin{split}\label{mi:ro*-F7}
\rho^{*}_{11}&=\rho^{*}_{22}=-\rho^{*}_{33}=-\rho^{*}_{44}
=-4\left(\lm_1\lm_3+\lm_2\lm_4\right)-6\mu_1\mu_2,\\
\rho^{*}_{55}&=4\mu_1\mu_2,\\
\rho^{*}_{13}&=\rho^{*}_{24}=\rho^{*}_{31}=\rho^{*}_{42}
=2\left(\lm^2_1+\lm^2_2-\lm^2_3-\lm^2_4\right)+3\left(\mu^2_1-\mu^2_2\right);\\
\end{split}
\\
\label{mi:tao*-F7}
\tau^{*}=-16\left(\lm_1\lm_3+\lm_2\lm_4\right)-20\mu_1\mu_2;
\\
\begin{split}\label{mi:ro*-fiKT-F7}
\ddot{\rho}^{*}_{11}&=\ddot{\rho}^{*}_{22}=-\ddot{\rho}^{*}_{33}=-\ddot{\rho}^{*}_{44}
=-4\left(\lm_1\lm_3+\lm_2\lm_4\right)-8\mu_1\mu_2,\\
\ddot{\rho}^{*}_{13}&=\ddot{\rho}^{*}_{24}=\ddot{\rho}^{*}_{31}=\ddot{\rho}^{*}_{42}
=2\left(\lm^2_1+\lm^2_2-\lm^2_3-\lm^2_4\right)+4\left(\mu^2_1-\mu^2_2\right);\\
\end{split}
\\
\label{mi:tao*-fiKT-F7}
\ddot{\tau}^{*}=-16\left(\lm_1\lm_3+\lm_2\lm_4\right)-32\mu_1\mu_2.
\end{gather}

Bearing in mind \eqref{mi:sec curv}, let $k_{ij}$ ($i\neq j$)
be the sectional curvature for  $\n$ (or, with respect to $R$)  of the
basic 2-plane $\al_{ij}$ with a basis $\{e_i,e_j\}$, where  $e_i, e_j\in\{e_1,\dots,
e_5\}$. The basic 2-planes
$\al_{ij}$ of $(G,\ff,\xi,\eta,g)$ are the following:
\begin{itemize}
\item totally real sections ---
$\al_{12}$, $\al_{14}$, $\al_{23}$, $\al_{34}$; %
\item $\ff$-holomorphic sections --- $\al_{13}$, $\al_{24}$; %
\item $\xi$-sections ---
$\al_{51}$, $\al_{52}$, $\al_{53}$, $\al_{54}$.
\end{itemize}

Then, using \eqref{mi:sec curv}, \eqref{mi:g} and \eqref{mi:R-F7}, we
compute the sectional curvatures $k_{ij}$ for $\n$ and obtain:
\begin{equation}\label{mi:sec.curv-R}
\begin{array}{c}
k_{12}=k_{34}=-\left(\lm^2_1+\lm^2_2-\lm^2_3-\lm^2_4\right)-3\mu^2_1, \\
k_{14}=k_{23}=-\left(\lm^2_1+\lm^2_2-\lm^2_3-\lm^2_4\right)+3\mu^2_2;\\
k_{13}=k_{24}=0;\\
k_{51}=k_{52}=k_{53}=k_{54}=\mu^2_1-\mu^2_2.
\end{array}
\end{equation}

Analogously, from  \eqref{mi:sec curv}, \eqref{mi:g} and
\eqref{mi:K-fB-F7} we
calculate the sectional curvatures $\ddot{k}_{ij}$ for $\ff$KT-connection $\ddot{D}$ and obtain:
\begin{equation}\label{mi:sec.curv-R-fiKT}
\begin{array}{c}
\ddot{k}_{12}=\ddot{k}_{34}=-\left(\lm^2_1+\lm^2_2-\lm^2_3-\lm^2_4\right)-4\mu^2_1,\\
\ddot{k}_{14}=\ddot{k}_{23}=-\left(\lm^2_1+\lm^2_2-\lm^2_3-\lm^2_4\right)+4\mu^2_2;\\
\ddot{k}_{13}=\ddot{k}_{24}=\ddot{k}_{51}=\ddot{k}_{52}=\ddot{k}_{53}=\ddot{k}_{54}=0.
\end{array}
\end{equation}

Let us consider the $\ff$B-connection $\dot{D}$ on
$(G,\ff,\xi,\eta,g)$ defined by \eqref{mi:Q}
 and
\eqref{mi:defn-fiB}. Then, by \eqref{mi:f} and \eqref{mi:nabli} we
compute its components as follows:
\begin{equation}\label{mi:fB-F7}
\begin{array}{ll}
\dot{D}_{e_1}e_1=-\dot{D}_{e_3}e_3=\lm_1e_2-\lm_3e_4,\quad
&\dot{D}_{e_5}e_1=-\mu_1e_2+\mu_2e_4, \\
\dot{D}_{e_1}e_2=-\dot{D}_{e_3}e_{4}=-\lm_1e_1+\lm_3e_3, \quad
&\dot{D}_{e_5}e_2=\mu_1e_1-\mu_2e_3,
\\
\dot{D}_{e_1}e_3=\dot{D}_{e_3}e_{1}=\lm_3e_2+\lm_1e_4, \quad
&\dot{D}_{e_5}e_3=-\mu_2e_2-\mu_1e_4,
\\
\dot{D}_{e_1}e_4=\dot{D}_{e_3}e_{2}=-\lm_3e_1-\lm_1e_3, \quad
&\dot{D}_{e_5}e_4=\mu_2e_1+\mu_1e_3,
\\
\dot{D}_{e_2}e_1=-\dot{D}_{e_4}e_{3}=\lm_2e_2-\lm_4e_4, \quad
&\dot{D}_{e_i}e_5=0, \; i\in\{1,2,3,4,5\}.
\\
\dot{D}_{e_2}e_2=-\dot{D}_{e_4}e_{4}=-\lm_2e_1+\lm_4e_3,&\\
\dot{D}_{e_2}e_3=\dot{D}_{e_4}e_{1}=\lm_4e_2+\lm_2e_4,&\\
\dot{D}_{e_2}e_4=\dot{D}_{e_4}e_{2}=-\lm_4e_1-\lm_2e_3,&
\end{array}
\end{equation}
Thus, we get that the basic non-zero components of the torsion of
$\dot{D}$ are:
\begin{equation*}\label{mi:TfB-F7}
\begin{split}
\dot{T}_{125}&=-\dot{T}_{215}=2\dot{T}_{251}
=-\dot{T}_{345}\\
&=2\dot{T}_{354}
=\dot{T}_{435}=-2\dot{T}_{521}=-2\dot{T}_{534}
=-2\mu_1,\\
\dot{T}_{145}&=-\dot{T}_{235}=2\dot{T}_{253}=\dot{T}_{325}\\
&=-\dot{T}_{415}=2\dot{T}_{451}=-2\dot{T}_{523}=-2\dot{T}_{541}
=-2\mu_2.
\end{split}
\end{equation*}

For the square norm
$\norm{\dot{T}}=g^{ip}g^{jq}g^{ks}\dot{T}_{ijk}\dot{T}_{pqs}$ we
obtain
\begin{equation}\label{mi:norm ddot{T}}
\norm{\dot{T}}=20\left(\mu^2_1-\mu^2_2\right).
\end{equation}

According to \eqref{mi:fB-F7}, we obtain the following components
$\dot{R}_{ijkl}=\dot{R}(e_i,e_j,e_k,e_j)$ of the curvature tensor
$\dot{R}$ of $\dot{D}$ on the manifold (the remaining ones are obtained, according to
\eqref{mi:L-curv-like-1} and \eqref{mi:L-curv-like-2}):
\begin{equation}\label{mi:RfB-F7}
\begin{split}
\dot{R}_{1212}&=-\dot{R}_{1234}=-\dot{R}_{3412}=\dot{R}_{3434}\\
&=\left(\lm^2_1+\lm^2_2-\lm^2_3-\lm^2_4\right)+2\mu^2_1,\\
\dot{R}_{1414}&=-\dot{R}_{1423}=-\dot{R}_{2314}=\dot{R}_{2323}\\
&=-\left(\lm^2_1+\lm^2_2-\lm^2_3-\lm^2_4\right)+2\mu^2_2,\\
\dot{R}_{1214}&=\dot{R}_{1412}=-\dot{R}_{1223}=-\dot{R}_{1434}=-\dot{R}_{2312}=\dot{R}_{2334}\\
&=-\dot{R}_{3414}=\dot{R}_{3423}=2\left(\lm_1\lm_3+\lm_2\lm_4\right)+2\mu_1\mu_2.
\end{split}
\end{equation}

Using \eqref{mi:f}, \eqref{mi:g} and \eqref{mi:RfB-F7}, we compute
the components of the Ricci tensor $\dot{\rho}$, the value of the
scalar curvature $\dot{\tau}$ and their associated
quantities for the $\ff$B-connection $\dot{D}$. The non-zero components of these tensors
and the scalar curvatures are the following:
\begin{gather}
\begin{split}\label{mi:rofB-F7}
\dot{\rho}_{11}&=\dot{\rho}_{22}=-\dot{\rho}_{33}=-\dot{\rho}_{44}
=-\dot{\rho}^{*}_{13}=-\dot{\rho}^{*}_{24}=-\dot{\rho}^{*}_{31}=-\dot{\rho}^{*}_{42}\\
&=-2\left(\lm^2_1+\lm^2_2-\lm^2_3-\lm^2_4\right)-2\left(\mu^2_1-\mu^2_2\right),\\
\dot{\rho}_{13}&=\dot{\rho}_{24}=\dot{\rho}_{31}=\dot{\rho}_{42}
=\dot{\rho}^{*}_{11}=\dot{\rho}^{*}_{22}=-\dot{\rho}^{*}_{33}=-\dot{\rho}^{*}_{44}\\
&=-4\left(\lm_1\lm_3+\lm_2\lm_4\right)-4\mu_1\mu_2;
\end{split}
\\
\dot{\tau}=-8\left(\lm^2_1+\lm^2_2-\lm^2_3-\lm^2_4\right)
-8\left(\mu^2_1-\mu^2_2\right),\label{mi:taofB-F7}\\
\dot{\tau}^{*}=-16\left(\lm_1\lm_3+\lm_2\lm_4\right)-16\mu_1\mu_2.\label{mi:tao*fB-F7}
\end{gather}

Taking into account \eqref{mi:sec curv}, \eqref{mi:g} and
\eqref{mi:RfB-F7}, we obtain the following basic sectional curvatures $\dot{k}_{ij}$ for the $\ff$B-connection:
\begin{equation}\label{mi:sec.curv-R-fiB}
\begin{array}{l}
\dot{k}_{12}=\dot{k}_{34}=-\left(\lm^2_1+\lm^2_2-\lm^2_3-\lm^2_4\right)-2\mu^2_1,\\
\dot{k}_{14}=\dot{k}_{23}=-\left(\lm^2_1+\lm^2_2-\lm^2_3-\lm^2_4\right)+2\mu^2_2;\\
\dot{k}_{13}=\dot{k}_{24}=\dot{k}_{51}=\dot{k}_{52}=\dot{k}_{53}=\dot{k}_{54}=0.
\end{array}
\end{equation}


Let us consider the $\ff$-canonical connection $\dddot{D}$ on
$(G,\ff,\xi,\eta,g)$ defined by \eqref{mi:T-can}. Then, by
\eqref{mi:f} and \eqref{mi:nabli} we compute its components as
follows:
\begin{equation}\label{mi:fcanon-F7}
\begin{split}
&\dddot{D}_{e_i}e_j=\dot{D}_{e_i}e_j, \quad i,j\in\{1,2,3,4\},\\
&\dddot{D}_{e_i}e_5=\dddot{D}_{e_5}e_i=0, \quad i\in\{1,2,3,4,5\}.
\end{split}
\end{equation}

The basic non-zero components of the torsion of $\dddot{D}$ are:
\begin{equation*}\label{mi:Tfcanon-F7}
\begin{split}
\dddot{T}_{125}&=-\dddot{T}_{215}=\dddot{T}_{251}=-\dddot{T}_{345}\\
&=\dddot{T}_{354}
=\dddot{T}_{435}=-\dddot{T}_{521}=-\dddot{T}_{534}=-2\mu_1,\\[4pt]
\dddot{T}_{145}&=-\dddot{T}_{235}=\dddot{T}_{253}=\dddot{T}_{325}\\
&=-\dddot{T}_{415}=\dddot{T}_{451}=-\dddot{T}_{523}=-\dddot{T}_{541}=-2\mu_2.
\end{split}
\end{equation*}

For the square norm $\norm{\dddot{T}}=g^{ip}g^{jq}g^{ks}\dddot{T}_{ijk}\dddot{T}_{pqs}$ we obtain
\begin{equation}\label{mi:norm dddot{T}}
\norm{\dddot{T}}=32\left(\mu^2_1-\mu^2_2\right).
\end{equation}

Using \eqref{mi:fcanon-F7} and \eqref{mi:RfB-F7}, we obtain that the
components $\dot{R}_{ijkl}$ of the curvature tensor $\dot{R}$ of
$\dot{D}$ and the components $\dddot{R}_{ijkl}$ of the curvature
tensor $\dddot{R}$ of  $\dddot{D}$ are equal on
$(G,\ff,\xi,\eta,g)$, \ie
\begin{equation}\label{Rd-ddd}
\dot{R}_{ijkl}=\dddot{R}_{ijkl},\quad
i,j,k,l\in\{1,\dots,5\}.
\end{equation}
Thus, it implies the following
\begin{prop}\label{prop:RR}
The curvature tensors of the $\ff$B-connection and the $\ff$-canon\-ical
connection are equal on $(G,\ff,\xi,\eta,g)$.
\end{prop}

\begin{prop}
The following characteristics are valid for $(G,\ff,\xi,\eta,g)$:
\begin{enumerate}
\item The $\ff$-holomorphic sectional curvatures for $\n$,
$\dot{D}$, $\ddot{D}$ and $\dddot{D}$ are zero;
\item The $\xi$-sectional curvatures for
$\dot{D}$, $\ddot{D}$ and $\dddot{D}$ are zero;
\item The associated Ricci tensor of $\n$ is proportional to the metric $g$
if and only if the following identities hold:
\[
\begin{array}{l}
\lm^2_1+\lm^2_2-\lm^2_3-\lm^2_4=\mu^2_1-\mu^2_2
=\mu_1\mu_2+\frac{2}{5}\left(\lm_1\lm_3+\lm_2\lm_4\right)=0;
\end{array}
\]
\item 
The manifold is scalar flat with respect to  $\n$,
$\dot{D}$, $\ddot{D}$ and $\dddot{D}$ if and only if the identities $\lm^2_1+\lm^2_2-\lm^2_3-\lm^2_4=\mu^2_1-\mu^2_2=0$ hold;
\item 
The manifold has vanishing associated scalar curvature
with respect to $\n$ (respectively, $\dot{D}$, $\ddot{D}$ and $\dddot{D}$)
if and only if
the identity
\[
\mu_1\mu_2=\nu\left(\lm_1\lm_3+\lm_2\lm_4\right)
\]
is valid for $\nu=-\frac{4}{5}$ (respectively, $\nu=-1$, $\nu=-\frac{1}{2}$ and $\nu=-1$);
\item 
The manifold has natural connections
$\dot{D}$, $\ddot{D}$ and $\dddot{D}$
coinciding with the Levi-Civita connection $\n$
if and only if $(G,\ff,\xi,\eta,g)$ belongs to $\F_0$;
\item 
The manifold is flat with respect to $\n$
(respectively, $\dot{D}$, $\ddot{D}$ and $\dddot{D}$)
if and only if it belongs to $\F_0$
    and the conditions $\lm^2_1+\lm^2_2-\lm^2_3-\lm^2_4=\lm_1\lm_3+\lm_2\lm_4=0$ hold;
\item 
The manifold has vanishing $\rho$, \ie it is Ricci-flat,(respectively,
$\rho^*$, $\dot{\rho}$, $\dot{\rho^*}$, $\ddot{\rho}$,
$\ddot{\rho^*}$, $\dddot{\rho}$, $\dddot{\rho^*}$ vanish)
    if and only if $(G,\ff,\xi,\eta,g)$ is flat, $R=0$.
\end{enumerate}
\end{prop}
\begin{proof}
Equation~(\ref{Rd-ddd}) implies that the components of the Ricci
tensor, the value of the scalar curvature (respectively, their
associated quantities) and the basic sectional curvatures of the
$\ff$B-connection and $\ff$-canonical connection are equal on
$(G,\ff,\xi,\eta,g)$.


The statements (1) and (2) are corollaries of \eqref{mi:sec.curv-R},
\eqref{mi:sec.curv-R-fiKT}, \eqref{mi:sec.curv-R-fiB} and \eqref{Rd-ddd}.

By virtue of \eqref{mi:ro*-F7} and \eqref{mi:g}, we get that $\rho^*$ is proportional to $g$
if and only if the identities in (3) are valid.

Using \eqref{mi:tao}, \eqref{mi:taofKT-F7}, \eqref{mi:taofB-F7}  and \eqref{Rd-ddd}, we obtain the
statement (4).

From \eqref{mi:tao*-F7}, \eqref{mi:tao*fB-F7},
\eqref{mi:tao*-fiKT-F7}  and \eqref{Rd-ddd} we get immediately the statement (5).

The statement (6) follow from
\eqref{mi:nabli}, \eqref{mi:fB-F7} and \eqref{mi:fcanon-F7} because of $\mu_1=\mu_2=0$.

Equation~(\ref{mi:R-F7}) and (6) imply (7).

The statement (8) follows from \eqref{mi:ro}, \eqref{mi:rofKT-F7},
\eqref{mi:ro*-F7}, \eqref{mi:ro*-fiKT-F7}, \eqref{mi:rofB-F7}, \eqref{Rd-ddd} and (7).
\end{proof}

\begin{prop}
The following conditions are equivalent:
\begin{enumerate}
\item The manifold $(G,\ff,\xi,\eta,g)$ is an
isotropic-$\F_0$-manifold;
\item The scalar curvatures for $\n$,
$\dot{D}$, $\ddot{D}$ and $\dddot{D}$ are equal;
\item The vectors
$\n_{e_i}\xi \; (i=1,2,3,4)$ are isotropic;
\item The Nijenhuis
tensor $N$ is isotropic;
\item The torsion tensors of $\dot{D}$,
$\ddot{D}$ and $\dddot{D}$ are isotropic;
\item The sectional
curvatures of the $\xi$-sections with respect to $\n$ vanish;
\item The equality $\mu_1=\pm\mu_2$ is valid.
\end{enumerate}
\end{prop}
\begin{proof}
According to \propref{mm:prop:Man31} the conditions (1), (3) and
(7) are equivalent. The other conditions are equivalent of
(7), bearing in mind \eqref{mi:tao}, \eqref{mi:taofKT-F7}, \eqref{mi:taofB-F7}, \eqref{Rd-ddd} for (2);
\eqref{mi:norm N} for (4); \eqref{mi:norm dot{T}}, \eqref{mi:norm ddot{T}},
\eqref{mi:norm dddot{T}} for
(5) and \eqref{mi:sec.curv-R} for (6).
\end{proof}


\end{document}